\documentclass[12pt]{amsart}
\usepackage{color}
\usepackage{amsmath}
\allowdisplaybreaks
\usepackage{amssymb}
\newtheorem{theorem}{Theorem}[section]
\newtheorem{corollary}[theorem]{Corollary}
\newtheorem{lemma}[theorem]{Lemma}
\newtheorem{proposition}[theorem]{Proposition}
\newtheorem{definition}[theorem]{Definition}

\newcommand{\C}{{\mathbb{C}}}

\renewcommand{\P}{{\mathbb{P}}}

\newcommand{\rank}{\mathrm{rank}}

\newcommand{\bbC}{\mathbb{C}}

\newcommand{\bbE}{\mathbb{E}}

\begin{document}
\title[A Brownian-Motion Approach to the Second Main Theorem]{A Brownian-Motion Approach to the Second Main Theorem for Meromorphic Mappings and Hypersurfaces with Truncated Counting Functions} 

\author{Nguyen Linh Chi$^1$}
\author{Si Duc Quang$^{1,2}$}
\address{$^1$Department of Mathematics, Hanoi National University of Education, 136-Xuan Thuy, Cau Giay, Hanoi, Vietnam}
\address{$^2$Institute of Natural Sciences, Hanoi National University of Education, 136-Xuan Thuy, Cau Giay, Hanoi, Vietnam}
\email{linhchihnue2401@gmail.com}
\email{quangsd@hnue.edu.vn}

\begin{abstract}
By using Brownian motion and stochastic calculus, we establish a second main theorem for holomorphic curves into a projective subvariety $V\subset\P^n(\C)$ with an arbitrary family $\mathcal Q$ of $q$ hypersurfaces $Q_1,\ldots,Q_q$ concerning its distributive constant $\Delta_{\mathcal Q,V}$. In our result, the counting functions are truncated to level $H_V(d)-1$, where $d=lcd(\deg Q_1,\ldots,\deg Q_d)$ and $H_V(d)$ is the Hilbert function of $V$. As an application of the second main theorem, we give a uniqueness theorem for holomorphic curves from $\C$ into $V$ sharing an arbitrary family of hypersurfaces regardless of multiplicity. 
\end{abstract}

\maketitle

\def\thefootnote{\empty}
\footnotetext{
2010 Mathematics Subject Classification:
Primary 32H30, 32A22, 30D35.\\
\hskip8pt Key words and phrases: Nevanlinna theory, second main theorem, holomorphic map, Brownian motion, Ito's integral, stochastic calculus.}

\section{Introduction}

In 1926, Nevanlinna \cite{N} established the value distribution theory for meromorphic functions on $\C$ by proving the second main theorem for meromorphic functions, with the targets being sets of values in $\bar{\C}$. Subsequently, H. Cartan \cite{C}  extended Nevanlinna theory to holomorphic mappings into projective space $\P^n(\C)$, with the targets being families of hyperplanes in general position. In 1985, by introducing the notion of Nochka weights for families of hyperplanes in subgeneral position, E.I. Nochka \cite{Nochka} generalized Cartan's result to the case of hyperplanes in subgeneral position.

In recent years, Nevanlinna theory has been further developed for the case of meromorphic mappings where the targets are families of hypersurfaces. Most of the methods used to establish the second main theorem in Nevanlinna theory rely on complex analysis, complex geometry, and potential theory. This analytic approach, however, still is restricted to mappings defined on complex affine spaces.

In 1986, T. K. Carne \cite{Car} initiated the study of the relationship between Brownian motion and Nevanlinna theory. Later on, A. Atsuji, in a series of publications \cite{A1,A2,A3,A4}, applied Brownian motion and stochastic calculus to the study of Nevanlinna theory for holomorphic mappings on general K\"{a}hler manifolds. More recently, X. Dong and his collaborators \cite{RCD,D1,D2,D3} have carried out significant research on using stochastic calculus and Brownian motion to establish the second main theorem for mappings defined on spaces more general than complex affine spaces. Although their results have not yet reached the strength of those in the case of $\C^n$, they nevertheless provide hope for a method that could generalize all known results in Nevanlinna theory for mappings defined on $\C^m$ to the setting of general K\"{a}hler manifolds.

Inspired by the work of X. Dong and his collaborators \cite{RCD}, together with the notion of distributive constant of an arbitrary family of hypersurfaces introduced by the second author \cite{Q21}, in this paper we will establish a new second main theorem based on Brownian motion and stochastic calculus for holomorphic curves from $\C$ into projective subvarieties, with arbitrary families of hypersurfaces where the counting functions are truncated by a level sufficiently small to allow the application of the uniqueness problem for meromorphic mappings. To present the results of this paper, we first recall the main concepts of Nevanlinna theory and some of the recent developments in the field.

Let $f:\C\to\P^n(\C)$ be a holomorphic map with reduced representation $\tilde f=(f_0,\cdots,f_n),$ where $f_0,\cdots,f_n$ are holomorphic functions on $\C$ without common zeros. The characteristic function $T_f(r)$ is defined by
$$T_f(r)=\int_0^{2\pi}\log\|\tilde f(r e^{i\theta})\|\frac{d\theta}{2\pi},$$
where $\|\tilde f\|=\sqrt{|f_0|^2+\cdots+ |f_n|^2} $. This definition is independent, up to an additive constant, of the choice of the reduced representation $\tilde f$ of $f.$  

Let $Q\in\C[x_0,\ldots,x_n]$ be a homogeneous polynomial of degree $d$ in variables $(x_0,\ldots,x_n)$. The proximity function of $f$ with respect to $Q$ is defined by
$$m_f(r,Q)=\int_0^{2\pi}\log\frac{\|\tilde f(re^{i\theta})\|^d \|Q\|}{|Q(\tilde f)(r e^{i\theta})|} \frac{d\theta}{2\pi},$$
where $\|Q\|$ is the sum of the absolute values of all coefficients of $Q$. This definition is also independent, up to an additive constant, of the choice of the reduced representation of $f$. 

Denote by $\nu_{Q(\tilde f)}$ the divisor of the holomorphic function $Q(\tilde f)$ on $\C$, i.e., $\nu_{Q(\tilde f)}(z)$ is the zero multiplicity of $Q(\tilde f)$ at $z$. For a positive integer $M$ ($M$ maybe $+\infty$), we define $n^{[M]}_f(r,Q)=\sum_{|z|<r}\min\{M,\nu_{Q(\tilde f)}(z)\}$. The counting function of $f$ with respect to $Q$ is defined by
$$N^{[M]}_f(r,Q) = \int_1^{r}\frac{n^{[M]}_f(t,Q)-n^{[M]}_f(0,Q)}{t}dt+n^{[M]}_f(0,Q)\log r.$$
By the Jensen's formula, we have
$$N_f (r,Q)=\int_0^{2\pi} \log |Q(\tilde f)(r e^{i\theta})| \frac{d\theta}{2\pi}+ O(1).$$

Let $D \subset \mathbb{P}^n(\mathbb{C})$ be a hypersurface of degree $d$, defined by a homogeneous polynomial $Q$ of degree $d$. Define 
$$m_f(r,D)=m_f(r,Q)\text{ and }N^{[M]}_f(r,D)=N^{[M]}_f(r,Q).$$
Thus, the first main theorem in Nevanlinna theory can be written as 
$$dT_f(r)=m_f(r,D)+N_f (r,D)+O(1).$$

Let $V \subset \mathbb{P}^n(\mathbb{C})$ be a projective subvariety of dimension $k$, 
and let $D_1,\ldots,D_q$ ($q \geq k+1$) be hypersurfaces in $\mathbb{P}^n(\mathbb{C})$ 
such that $V \not\subset D_j$ for $1 \leq j \leq q$.
 The hypersurfaces $D_1,\ldots, D_q$ are said to be in $N$-subgeneral position with respect to $V\ (k\le N\le q-1)$ if 
$$D_{j_0}\cap\cdots\cap D_{j_N}\cap V=\varnothing$$ for every $1\leq j_0<\cdots<j_N\leq q.$ If $N=k$ then we say that $D_1,\cdots, D_q$ are in general position with respect to $V$.

In 2004, M. Ru \cite{Ru04} established a second main theorem for algebraically non-degenerate holomorphic maps from $\mathbb{C}$ into $\mathbb{P}^n(\mathbb{C})$, 
with respect to families of hypersurfaces in general position. Later on, in 2009, Ru \cite{Ru09}  generalized his result by considering the case of algebraically nondegenerate holomorphic maps into a projective subvariety $V\subset\mathbb{P}^n(\mathbb{C})$, as follows.

\vskip0.2cm
\noindent
\textbf{Theorem A} (see \cite[The main result]{Ru09}) {\it Let $V \subset \mathbb{P}^n(\mathbb{C})$ be a smooth projective subvariety of dimension $k \geq 1$, 
and let $D_1,\dots,D_q \subset \mathbb{P}^n(\mathbb{C})$ be hypersurfaces in general position with respect to $V$. Let $f$ be an algebraically nondegenerate holomorphic map from $\C$ into $V$. Then for every $\epsilon > 0$,
$$\bigl\|\ (q-n-1-\epsilon) T_f(r)\le\sum_{i=1}^q\frac{1}{\deg D_i}N_f(r,D_i).$$}
Here, the notation $\| P$ means that the assertion $P$ holds for all $r\in (0,\infty)$ outside a subset of finite linear measure. 

Theorem A had been generalized and extended by many authors. We refer the reader to several important works \cite{AP10,DT11,DT20,G,Q19, YYJ,XC} concerning the establishment of second main theorems in Nevanlinna theory. Recently, in \cite{Q21} the author introduced the notion of the distributive constant of a family of hypersurfaces as follows.   

\vskip0.2cm
\noindent
{\bf Definition B} (see \cite[Definition 1.1]{Q21}). {\it Let $V$ be a projective subvariety of $\P^n(\C)$ of dimension $k$, and let $D_1,\ldots,D_q$ be $q$ hypersurfaces in $\P^n(\C)$. The distributive constant of the family $\{D_1,\ldots,D_q\}$ with respect to $V$ is defined by
$$ \Delta:=\underset{\varnothing\ne\Gamma\subset\{1,\ldots,q\}}\max\dfrac{\sharp\Gamma}{k-\dim\left (\bigcap_{j\in\Gamma}D_j\right )\cap V}.$$
Here, we note that $\dim\varnothing =-\infty$.} 

Using the notion of distributive constant, in \cite{Q21} the second author proved the following second main theorem.

\vskip0.2cm
\noindent
\textbf{Theorem C} (see \cite[Theorem 1.2]{Q21}) {\it Let $V$ be a smooth projective subvariety of dimension $k\ge 1$ of $\P^n(\C)$. Let $\{D_1,\ldots,D_q\}$ be a family of hypersurfaces of $\P^n(\C)$ with distributive constant $\Delta$ with respect to $V$ and $\deg D_i=d_i\ (1\le i\le q)$. Let $f$ be an algebraically nondegenerate holomorphic map from $\C$ into $V$. Then, for every $\epsilon >0$, we have
$$\bigl \|\ \left (q-\Delta (k+1)-\epsilon\right) T_f(r)\le\sum_{i=1}^q\frac{1}{d_i}N^{[M_0]}_f(r,D_i),$$
where $M_0=\left[d^{k^2+k}\deg (V)^{n+1}e^k\Delta^k(2k+4)^k(k+1)^k(q!)^k\epsilon^{-k}\right]$ and $d$ is the least common multiple of $d_1,\ldots,d_q$, i.e., $d=lcm(d_1,\ldots,d_q).$}

However, in all the above mentioned second main theorems, the truncation levels for counting functions are disregarded or are very large. As a consequence, their effectiveness in applications to related problems, such as the unicity problem for holomorphic curves, is substantially constrained. Actually, the earliest second main theorem for holomorphic maps from $\C$ into $\P^n(\C)$ and a family of hypersurfaces with truncated counting functions (with a small enough truncation level) was given by D. P. An, S. D. Quang and D. D. Thai  \cite{TAQ} in 2013. Later on, in \cite{QA}, Quang and An generalized that result by proving a second main theorem for a holomorphic curve from $\C$ into a projective subvariety $V$ of $\P^n(\C)$ and a family of hypersurfaces in $N$-subgeneral position with respect to $V$. In order to state the result of Quang and An, we recall the following.

Let $V$ be as above. We denote by $I(V)$ the ideal of homogeneous polynomials in $\C[x_0,\ldots,x_n]$ defining $V$ and by $H_d$ the $\mathbb C$-vector space of all homogeneous polynomials in $\mathbb C [x_0,...,x_n]$ of degree $d.$  Define 
$$I_d(V):=\dfrac{H_d}{I(V)\cap H_d}\text{ and }H_V(d):=\dim I_d(V).$$
Then $H_V(d)$ is called the Hilbert function of $V$. Each element of $I_d(V)$ which is an equivalence class of an element $Q\in H_d,$ will be denoted by $[Q]$, 

For a holomorphic map $f$ from $\C$ into $V$, we say that $f$ is degenerate over $I_d(V)$ if there is $[Q]\in I_d(V)\setminus \{0\}$ such that $Q(\tilde f)\equiv 0$ for a reduced representation $\tilde f$ of $f$. Otherwise, we say that $f$ is nondegenerate over $I_d(V)$. The result of Quang and An is stated as follows.

\vskip0.2cm
\noindent
{\bf Theorem D} (see \cite[Theorem 1.1]{QA}). {\it Let $V$ be a complex projective subvariety of $\P^n(\C)$ of dimension $k\ (k\le n)$. Let $\{D_j\}_{j=1}^q$ be hypersurfaces of $\P^n(\C)$ in $N$-subgeneral position with respect to $V$ with $\deg D_j=d_j\ (1\le j\le q)$. Let $d=lcm (d_1,...,d_q)$. Let $f$ be a holomorphic map from $\C$ into $V$ which is nondegenerate over $I_d(V)$. Then, we have
$$\biggl \|\ \left(q-\dfrac{(2N-k+1)H_{V}(d)}{k+1}\right )T_f(r)\le \sum_{j=1}^{q}\dfrac{1}{d_j}N^{[H_{V}(d)-1]}_{f}(r,D_j)+o(T_f(r)).$$}

Our aim in this paper is to establish a second main theorem for a holomorphic map from $\C$ into a projective subvariety $V$ of $\P^n(\C)$ with respect to an arbitrary family of hypersurfaces and its distributive constant. Our main result is stated as follows.

\begin{theorem}\label{1.1}
Let $V$ be a projective subvariety of $\P^n(\C)$ and let $\{D_1,\ldots,D_q\}$ be a family of $q$ hypersurfaces of $\P^n(\C)$ which has distributive constant $\Delta$ with respect to $V$. Let $M=H_V(d)-1$ and $d=lcm(\deg D_1,\ldots,\deg D_q)$. Let $f$ be a holomorphic map from $\C$ into $V$ which is nondegenerate over $I_d(V)$. Then, for any positive $\epsilon$ and $\delta$, we have
$$\|\,(q-\Delta(M+1+\epsilon))T_f(r)\le\sum_{j=1}^q\frac{1}{d}N^{[M]}_f(r,D_j)+\Delta\delta \log r + O(1).$$
\end{theorem}

Remark. Let $N=\ell k, q=\ell p\ (p\ge k+1)$ and $\{D_1,\ldots,D_q\}$ be a family of hypersurfaces with $D_{tp+s}\equiv D_{s}$ for every $1\le t\le\ell-1,1\le s\le p$ and $\{D_1,\ldots,D_p\}$ being a family of hypersurfaces in general position w.r.t. $V$. In this case $\{D_1,\ldots,D_q\}$ is in $(\ell k+1)$-subgeneral position w.r.t to $V$ and has distributive constant $\Delta$ w.r.t to $V$ equal to $\ell$. Obviously that $\Delta=\ell\le \frac{(2\ell k-k+1)}{k+1}$. Therefore, Theorem \ref{1.1} is new and maybe better than Theorem D in some case. 

As an application of Theorem \ref{1.1},  the second aim of this paper is to give a uniqueness theorem for holomorphic maps of $\C$ into $V$ sharing a few hypersurfaces without counting multiplicity. Our uniqueness theorem is stated as follows.

\begin{theorem}\label{1.2}
Let $V$ be a projective subvariety of $\mathbb P^n(\mathbb C)$.  Let $\{D_i\}_{i=1}^q$ be a family of $q$ hypersurfaces in $\mathbb P^n(\mathbb C)$ with the distributive constant $\Delta$ with respect to $V$. Let $d=lcm (deg D_1,...,deg D_q)$. Let $f$ and $g$ be holomorphic maps from $\C$ into $V$ which are nondegenerate over $I_d(V)$.  Assume that $f=g \text{ on }\bigcup_{i=1}^q(f^{-1}(D_i)\cup g^{-1}(D_i)).$
\begin{itemize}
    \item[(a)] If $q>\Delta\left(\frac{2k(H_V(d)-1)}{d}+ H_V(d)\right),$ then $f=g.$
    \item[(b)] Suppose further that $f^{-1}(D_i\cap D_j)=\varnothing$ for every $1\le i< j\le q.$ If $q>\frac{2(H_V(d)-1)}{d}+\Delta H_V(d),$ then $f=g.$
\end{itemize}
\end{theorem}

\section{Probability theory}
We review certain notions and auxiliary results in stochastic calculus due to \cite{RCD}.
\subsection*{2.1 Expectation of random variables}
A probability space is a triple ($\Omega, \mathcal{F}, \mathrm{P}$), where $\Omega$ is a non-empty set, $\mathcal{F}$ is a $\sigma$-algebra of subsets of $\Omega$ and $P$ is a probability measure on $(\Omega, \mathcal{F})$. A real (resp. complex) random variable $X$ is a measurable function $X: \Omega \rightarrow \mathbb{R}$ (resp. $\mathbb{C}$ ). The expectation of $X$ is defined as
$$\mathrm{E}[X]:=\int_{\Omega} X(w) d \mathrm{P}(w).$$
\begin{lemma}[{see \cite[Proposition 1.3]{B}}]\label{2.1}
If $g$ is convex and $X$ and $g(X)$ are integrable, where $X$ is a random variable, then
$$g(\mathrm{E}[X]) \le \mathrm{E}[g(X)].$$
\end{lemma}

Let $\mathcal{G} \subset \mathcal{F}$ be a sub $\sigma$-algebra. The conditional expectation, denoted by $\mathrm{E}[X \mid \mathcal{G}]$, is an random variable $Y$ such that $Y$ is $\mathcal{G}$-measurable and $\int_{A} Y d \mathrm{P}=\int_{A} X d \mathrm{P}$, i.e., $\mathrm{E}[Y ; A]=\mathrm{E}[X ; A]$, for any $A \in \mathcal{G}$.

\subsection*{2.2 Stochastic process and stopping times}

Let $(\Omega,\mathcal{F},\mathbb{P})$ be a probability space. A \emph{filtration} $\{\mathcal{F}_{t}\}_{t \in I}$ is a family of sub-$\sigma$-algebras $\mathcal{F}_{t} \subseteq \mathcal{F}$, indexed by $I$ (typically $I = \mathbb{R}^{+}$ or $I = \mathbb{N}$), such that $\mathcal{F}_{s} \subseteq \mathcal{F}_{t}$ whenever $s < t$. In this paper, we always assume that the filtration is \emph{right-continuous}, meaning 
$$\mathcal{F}_{t+} := \bigcap_{s > t} \mathcal{F}_{s} = \mathcal{F}_{t}, \quad \forall\, t \in I.$$

A \emph{stochastic process} is a family of random variables $\{X_{t}\}_{t \in I}$. The process is called \emph{discrete-time} if $I = \mathbb{N}$, and \emph{continuous-time} if $I = \mathbb{R}^{+}$ and, for almost every $\omega \in \Omega$, the sample path 
$$t\in[0, \infty)  \mapsto X_{t}(\omega)\in \mathbb R\ (\text{or }\in\C) $$
is continuous.

The process is said to be \emph{adapted} to $\{\mathcal{F}_{t}\}_{t \in I}$ if $X_{t}$ is $\mathcal{F}_{t}$-measurable for each $t \in I$.  
In this paper, we only consider continuous processes. A nonnegative random variable $T$ is called a \emph{stopping time} with respect to $\{\mathcal{F}_{t}\}$ if 
$$\{T \geq t\} \in \mathcal{F}_{t}, \quad \forall\, t \in I.$$
For a process $\{X_{t}\}$ and a stopping time $T$, we define a random variable 
$$X_{T}(\omega) := X_{T(\omega)}(\omega), \quad \omega \in \{T < \infty\}.$$

\subsection*{2.3 Martingales and Brownian motion}
A martingale is a stochastic process $\{M_t\}_{t \in I}$ that is adapted to a given filtration $\{\mathcal{F}_t\}$ and satisfies
$$\mathbb{E}[M_t \mid \mathcal{F}_s] = M_s, \quad \forall\, t \geq s, \; t,s \in I.$$

\begin{definition}[Brownian motion]\label{2.2}
A Brownian motion in $\mathbb{R}^{d}$ is a stochastic process $\left(B_{t}\right)_{t \ge 0}$ with values in $\mathbb{R}^{d}$ which satisfies the following properties
\begin{itemize}
\item[(i)] For any $0 \le t_{1}<\cdots<t_{n}$, the random variables $B_{t_{1}}, B_{t_{2}}-B_{t_{1}}$, $\ldots, B_{t_{n}}-B_{t_{n-1}}$ are independent.
\item[(ii)] For any $t>s \ge 0$, the random variables $B_{t}-B_{s}$ satisfies the normal distribution $N(0, t-s)$, i.e., $\mathrm{P}\left\{\left(B_{t}-B_{s}\right) \in d V(x)\right\}=p(t-s, x) d V(x)$, where $p(t, x)=\frac{1}{(2 \pi t)^{d / 2}} \exp\left(-\frac{\|x\|^{2}}{2 t}\right)$ and $d V$ is an infinitesimal volume element for $\mathbb{R}^{d}$.
\item[(iii)] $B_{t}(w)$ is continuous in $t$ for a.s. $w \in \Omega$.
\end{itemize}
\end{definition}
As a stochastic process, the Brownian motion is a martingale. According to \cite{B}, the Brownian motion on $\C$ always exists.

Let $\{X_{t}\}_{t\ge 0}$ be a Brownian motion on $\C$ started at $0$. For $x\in\C$, denote by $\mathrm{P}^x$ and $E_x$ the probability and expectation associated to the Brownian motion $\{(X_{t}+x)\}_{t\ge 0}$. Let $g_{r}(0, z)$ be the Green function of the disc $D_{r}:=\{z|\ |z|<r\}$ with the Dirichlet boundary condition. By \cite[Subsection 7.4]{I}, we have
\begin{align}\label{1}
g_{r}(0, z) d V(z)=\frac{1}{2} E_{0}\left(\text {times } X_{t} \text { spends in } d V(z) \text { before } \tau_{r}\right) 
\end{align}
for an infinitesimal area $d V(z)$ about $z \in D_{r}$, where $\tau_{r}=\inf \left\{t>0 \mid X_{t} \notin D_{r}\right\}$ is the stopping time. Thus, for any bounded continuous function $\psi$ on $D_{r}$, we have the co-area formula as follows
\begin{align}\label{2}
\mathrm{E}_{0}\left[\int_{0}^{\tau_{r}} \psi\left(X_{t}\right) d t\right]=\int_{D_{r}} g_{r}(0, y) \psi(y) d V(y). 
\end{align}

\begin{lemma}[{see \cite[Lemma 3.2]{RCD}}]\label{2.3}
 Let $u$ be a function on $\mathbb{C}$ with mild singularities. Assume that $x$ is not a singular point of $\log |u|$. Then, for any stopping time $T$ such that
$$\left.\mathrm{E}_{x}\left[\int_{0}^{T}(\Delta \log |u|)\left(X_{s}\right)\right) d s\right]<\infty$$
we have
$$\mathrm{E}_{x}\left[\log |u|\left(X_{T}\right)\right]-\log |u|(x)+N_{x}(T, \log |u|)=\frac{1}{2} \mathrm{E}_{x}\left[\int_{0}^{T}(\Delta \log |u|)\left(X_{s}\right) d s\right],$$
where
$$N_{x}(T, v):=\lim _{r \rightarrow \infty} \lambda \mathrm{P}^{x}\left(\sup _{0 \le s \le T} v^{+}\left(X_{s}\right)>\lambda\right)-\lim _{r \rightarrow \infty} \lambda \mathrm{P}^{x}\left(\sup _{0 \le s \le T} v^{-}\left(X_{s}\right)>\lambda\right)$$
with $v^{+}=\max \{0, v\}$ and $v^{-}=-\min \{0, v\}$. In particular, if $T=\tau_{r}$ and $x=0$, then
$$\mathrm{E}_{0}\left[\log |u|\left(X_{\tau_{r}}\right)\right]-\log |u|(0)+N_{u}(r, \infty)-N_{u}(r, 0)=\frac{1}{2} \mathrm{E}_{0}\left[\int_{0}^{\tau_{r}}(\Delta \log |u|)\left(X_{s}\right) d s\right].$$
\end{lemma}

\begin{lemma}[{see \cite[Lemma 3.5]{RCD}}]\label{2.4}
Let $u$ be a non-negative, locally integrable function on $D$. Assume $u$ is bounded on a neighborhood of the origin. Let $\tau_{r}:=\inf \left\{t>0,\left|X_{t}\right| \ge r\right\}$ be the stopping time. Then, for any $\delta>0$,
$$ \|\, \log \mathrm{E}_{0}\left[u\left(X_{\tau_{r}}\right)\right] \leqslant(1+\delta)^{2} \log \mathrm{E}_{0}\left[\int_{0}^{\tau_{r}} u\left(X_{s}\right) d s\right]+\delta \log r.$$
\end{lemma}

\section{The theory for holomorphic curves}

We will follow the methods presented in \cite{RCD} and \cite{Q24} to establish the sum into product inequality (see Theorem \ref{3.5}) and then prove a second main theorem for holomorphic curves and hypersurfaces concerning the Wronskian of the curves (see Theorem \ref{3.7}).

\subsection*{3.1 Associated curves}

Let $f : \C \to \mathbb{P}^n(\mathbb{C})$ be a linearly nondegenerate holomorphic map with a reduced representation $\tilde f=(f_0,\ldots,f_n)$. 
Consider the holomorphic maps $\tilde F_k$ defined by
$$\tilde F_k = \tilde f \wedge\tilde f' \wedge \cdots \wedge\tilde  f^{(k)} : \mathbb{C} \to \bigwedge^{k+1} \mathbb{C}^{n+1}.$$
Since $f$ is linearly nondegenerate, $\tilde F_k \neq 0$ for $0 \le k \le n$.  
We define the $k$-th associated map by
$$F_k = P(\tilde F_k) : \mathbb{C} \to \mathbb{P}(\bigwedge^{k+1} \mathbb{C}^{n+1})=\mathbb{P}^N(\mathbb{C}),$$ 
where $N = \binom{n+1}{k+1} - 1$ and $P$ is the natural projection.

The $k$-th characteristic function is defined by
$$T_{F_k}(r) = \int_0^r \frac{dt}{t} \int_{|z|<t} F_k^* \omega,$$
where $\omega$ is the Fubini-Study form on $\mathbb{P}^N(\mathbb{C})$, and
$$T_f(r) = T_{F_0}(r).$$

\begin{lemma}[{see \cite[Lemma 4.2]{RCD}}]\label{3.1}
Let $\delta > 0$. Then, for any $0 \le k \le n$,
$$\big\|\, N_{F_k}(r, 0) + T_{F_k}(r) \le (2n+1) T_f(r) + O(\log T_f(r)) + \delta \log r,$$
where $N_{F_k}(r, 0)$ is the counting function of the zero divisor of $F_k$.
\end{lemma}

\subsection*{3.2 Sum into Product inequality}

Let $f$ be a holomorphic map from $\C$ into a projective subvariety $V$ of dimension $k$ of $\P^n(\C)$. Let $d$ be a positive integer and assume that $f$ is nondegenerate over $I_d(V)$. We fix a $\C$-ordered basis $\mathcal V=([v_0],\ldots,[v_M])$ of $I_d(V)$, where $v_i\in H_d$ and $M=H_V(d)-1$.  

Let $\tilde f=(f_0,\ldots,f_n)$ be a reduced representation of $f$. Consider the holomorphic map
$$ F=(v_0(\tilde f),\ldots,v_M(\tilde f))$$
and 
$$F_{p} = (F_{p})_z := F^{(0)} \wedge F^{(1)} \wedge \cdots \wedge F^{(p)}:\C\rightarrow \bigwedge_{p+1}\C^{M+1}$$
for $0\le p\le M$, where 
\begin{itemize}
\item $F^{(0)}:=F=(v_0(\tilde f),\ldots,v_M(\tilde f))$,
\item $F^{(l)}:=\left (v_0(\tilde f)^{(l)},\ldots, v_M(\tilde f)^{(l)}\right)$ for each $l=0, 1,\ldots , p$,
\item $v_i(\tilde f)^{(l)} \ (i =0,\ldots, M)$ is the $l^{th}$-derivatives of $v_i(\tilde f)$.
\end{itemize}
The norm of $F_{p}$ is given by
$$|F_{p}|:=\left (\sum_{0\le i_0<i_1<\cdots<i_p\le M}\left |W(v_{i_0}(\tilde f),\ldots,v_{i_p}(\tilde f))\right|^2\right)^{1/2}, $$
where 
$$W(v_{i_0}(\tilde f),\ldots,v_{i_p}(\tilde f)):=\det\left (v_{i_j}(\tilde f)^{(l)}\right)_{0\le l,j\le p}$$ 
%
We use the same notation $\langle,\rangle$ for the canonical hermitian product on $\bigwedge^{l+1}\C^{M+1}\ (0\le l\le M)$. For two vectors $A\in \bigwedge^{k+1}\C^{M+1}\ (0\le k\le M)$ and $B\in\bigwedge^{p+1}\C^{M+1}\ (0\le p\le k)$, there is one and only one vector $C\in\bigwedge^{k-p}\C^{M+1}$ satisfying 
$$ \langle C,D\rangle=\langle A,B\wedge D\rangle\ \forall D\in \bigwedge^{k-p}\C^{M+1}.$$
The vector $C$ is called the interior product of $A$ and $B$, and denoted by $A\vee B$.

Now, for a nonzero homogeneous polynomial $Q$ of degree $d$ in $H_d$, we have
$$[Q]=\sum_{i=0}^Ma_i[v_i].$$
Hence, we associate $Q$ with the vector $H=(a_0,\ldots,a_M)\in\C^{M+1}$ and define $F_{p}(Q)=F_{p}\vee H$. Then, we may see that
\begin{align*}
F_{0}(Q)&=a_0v_0(F)+\cdots+a_Mv_M(F)=Q(F),\\ 
|F_{p}(Q)|&=\left (\sum_{0\le i_1<\cdots<i_p\le M}\sum_{\ell\ne i_1,\ldots,i_p}a_l\left |W(v_\ell(F),v_{i_1}(F),\ldots,v_{i_p}(F))\right|^2\right)^{1/2}
\end{align*}
Finally, for $0\le p\le M$, the $p^{th}$-contact function of $f$ for $Q$ (with respect to $\mathcal V$) is defined (not depend on the choice of the local coordinate) by
$$\varphi_{p}(Q):=\dfrac{|F_{p}(Q)|^2}{|F_{p}|^2}.$$

For each $p\ (0\le p\le M-1)$, let $M_p=\binom{M+1}{p+1}-1$ and $\pi_p$ be the canonical projection from $\bigwedge^{p+1}\C^{M+1}\sim\C^{M_p+1}$ onto $\P^{M_p}(\C)$. Denote by $\Omega_p$ the pullback of the Fubini-Study form on $\P^{M_p}(\C)$ by the map $\pi_p\circ F_{p}$, i.e., $\Omega_p = dd^c\log |F_{p}|^2$. Let
$$\Omega_p = h_p \, dz \wedge d\bar{z}, \quad 0 \le p \le M.$$
Then $h_p$ is non-negative. We have the following lemma due to \cite{Ru01}.  
\begin{lemma}[{see \cite[Lemma 4.1]{Ru01}}]\label{3.2}
$$h_p = \frac{|F_{p-1}|^2 \cdot |F_{p+1}|^2}{|F_p|^4}$$
for $0 \le p \le M$, and by convention $|F_{-1}| = 1$.
\end{lemma}

\begin{proposition}[{see \cite[Proposition 2.5.1]{Fu93}}]\label{3.3}
Let $V,d$ and $M$ be as above. For each positive $\epsilon$ there exists a constant $\delta_0(\epsilon)$, depending only on $\epsilon$, such that for any nonzero homogeneous polynomial $Q$ of degree $d$ in $H_d$ and any constant $\delta>\delta_0(\epsilon)$
$$ dd^c\log\dfrac{1}{\log(\delta/\varphi_{p}(Q))}\ge\dfrac{\varphi_{p}(Q)}{\varphi_{p+1}(Q)\log(\delta/\varphi_{p}(Q))}\Omega_p-\epsilon\Omega_p.$$
\end{proposition}
%
Let $Q_1,\ldots,Q_q$ be $q$ nonzero homogeneous polynomials in $\C[x_0,\ldots,x_n]\setminus I(V).$ Suppose that each $Q_i$ define a hypersurface $D_i$ in $V\subset\P^n(\C)$ for every $1\le i\le q$. We also say that $\{Q_1,\ldots,Q_q\}$ has distributive constant $\Delta$ with respect to $V$ if $\{D_1,\ldots,D_q\}$ has distributive constant $\Delta$ with respect to $V$.

\begin{theorem}\label{3.4}
Let $V,d$ and $M$ be as above. Let $f:\C\rightarrow V\subset\P^n(\C)$ be a holomorphic curve and let $\mathcal Q=\{Q_1,\ldots, Q_q\}$ be a family of hypersurfaces of the same degree $d$ in $\P^n(\C)$ with distributive constant $\Delta$, where $q>\Delta(M+1)$. Assume that $f$ is nondegenerate over $I_d(V)$ and has a reduced representation $\tilde f=(f_0,\ldots,f_n)$. For an arbitrarily given $\delta >1$ and $0\le p\le M-1$, we set
$$\Phi_{jp}:=\dfrac{\varphi_{p+1}(Q_j)}{\varphi_{p}(Q_j)\log^2\left (\delta/\varphi_{p}(Q_j)\right)}.$$
Then, there exists a positive constant $C$ depending only on $p$ and $Q_j\ (1\le j\le q)$ such that
$$ \sum_{j=1}^q\Phi_{jp}\ge C\left (\prod_{j=1}^q\Phi_{jp}\right)^{1/\Delta(M-p)}$$
holds on $S-\bigcup_{1\le j\le q}\{z;\varphi_{p}(Q_j)(z)=0\}$.
\end{theorem}
\begin{proof}
For each $j\ (1\le j\le q)$, we write
$$ [Q_j]=\sum_{i=0}^Ma_{ji}[v_i] $$
and set $a_j=(a_{j0},\ldots,a_{jM})\in\C^{M+1}$. Let $\mathcal R_p$ be the set of all subsets $R$ of $Q=\{1,2,\ldots, q\}$  such that $\rank_{\C}\{a_j\ |\ j\in R\}\le M-p$. For each $P\in G(M,p)$ (the Grassmannian manifold of all $(p+1)$-dimension sublinear spaces of $\C^{M+1}$), take a decomposable $(p+1)$-vector $E$ such that
$$ P=\{X\in\C^{M+1}; E\wedge X=0\}$$
and set
$$ \psi_p(P)=\underset{R\in\mathcal R_p}{\max}\min\left\{\dfrac{\left|E\vee a_j\right|^2}{|E|^2};j\not\in R\right\}. $$
Then $\psi_p(P)$ depends only on $P$. We may regard $\psi_p$ as a function on the Grassmannian manifold $G(M,p)$. For each nonzero $(k+1)$-vector $E=E_0\wedge E_1\wedge\ldots\wedge E_k$ we set
$$R=\{j\in Q; E\vee a_j=0\}.$$
Note that $E\vee a_j=0$ if and only if $a_j$ is contained in the orthogonal complement of the vector space $\mathrm{Span}(E_0,\ldots, E_k)$. Then we see
$$\rank_{\C}\{a_j\ |\ j\in R\}=\dim\mathrm{Span}(a_j; j\in R)\le M - p,$$
namely, $R\in\mathcal R_p$. This yields that $\psi_p$ is positive everywhere on $G(M,p)$.
Since $\psi_p$ is obviously continuous and $G(M,p)$ is compact, we can take a positive constant $\delta$ such that $\psi_p(P)>\delta$ for each $P\in G(M,p)$.

Take a point $z$ with $F_{p}(z)=0$. The vector space generated
by $F^{(0)}(z), F^{(1)}(z)$, ..., $F^{(p)}(z)$ determines a point in $G(M,p)$. Therefore, there is a set $R$ in $\mathcal R_p$ with 
$$ \rank_{\C}\{a_j\ |\ j\in R\}\le M-p $$ 
such that $\varphi_{p}(Q_j)(z)\ge\delta$ for all $j\not\in R$. Then, we can choose a finite positive constant $K$ depending only on $Q_j\ (1\le j\le q)$ such that $\Phi_{jp}(z)\le K$ for all $j\not\in R$. Set
$$T:= \{j| \Phi_{jp}(z)> K\}.$$

We consider the following two cases.

\textit{Case 1: }Assume that $T$ is an empty set. We have
\begin{align*}
\sum_{j=1}^q\Phi_{jp}&\ge  q\left (\prod_{j=1}^q\Phi_{jp}\right)^{\frac{1}{q}}\ge C_1\left (\prod_{j=1}^q\frac{\Phi_{jp}}{K}\right)^{\frac{1}{\Delta(M-p)}},
\end{align*}
because $q\ge\Delta(M+1)>\Delta(M-p)$. Here $C_1$ is a  positive constant which depends only on $Q_1,\ldots,Q_q$.

\textit{Case 2: }Assume that $T\ne\emptyset$. We see that $T\subset R$ and so $\rank_{\C}\{a_j\ |\ j\in T\}\le M-p$ holds. It yields that $\sharp T\le\Delta(M-p)$. Then, we have
\begin{align*}
\sum_{j=1}^q\Phi_{\mathcal V,jp}&\ge\sum_{j\in T}\Phi_{jp}\ge \sharp TK\prod_{j\in T}\left(\frac{\Phi_{jp}}{K}\right)^{1/\sharp T}\ge C_2\left(\prod_{j=1}^q\frac{\Phi_{jp}}{K}\right)^{1/\Delta(M-p)},
\end{align*}
for some positive constant $C_2>0$, which depends only on $Q_1,\ldots,Q_q$.

From the above two cases, we get the conclusion of the theorem.
\end{proof}

\begin{theorem}\label{3.5}
Let the notations and the assumption be as in Theorem \ref{3.4}. Then, for every $\epsilon>0$, there exist a positive number $\delta\ (>1)$ and a positive constant $C$, depending only on $\epsilon$ and $Q_j\ (1\le j\le q)$ such that
\begin{align*}
dd^c&\log\dfrac{\prod_{p=0}^{M-1}|F_{p}|^{2\epsilon}}{\prod_{1\le j\le q,0\le p\le M-1}\log^{2/\Delta}\left(\delta/\varphi_{p}(Q_j)\right)}\\
&\ge C\left (\dfrac{|F_{0}|^{2\left (\frac{q}{\Delta}-M-1\right)}|F_{M}|^2}{\prod_{j=1}^q(|F_{0}(Q_j)|^2\prod_{p=0}^{M-1}\log^2(\delta/\varphi_{p}(Q_j)))^{1/\Delta}}\right)^{\frac{2}{M(M+1)}}dd^c|z|^2.
\end{align*}
\end{theorem}

\begin{proof}
We denote by $A$ the left-hand side. Then we have
$$A=\epsilon\sum_{p=0}^{M-1}\Omega_p+\sum_{j=1}^q \sum_{p=0}^{M-1}dd^c\log\dfrac{1}{\log^{2}\left (\delta/\varphi_{p}(Q_j)\right)}.$$
Choose a positive number $\delta_0(\epsilon/\ell)$ with the properties as in Proposition \ref{3.3}, where $\ell=q/\Delta$. For an arbitrarily fixed $\delta\ge\delta_0(\epsilon/\ell)$, we obtain
\begin{align*}
A&\ge\epsilon\sum_{p=0}^{M-1}\Omega_p+\sum_{j=1}^q\sum_{p=0}^{M-1}\left (\dfrac{2\varphi_{p+1}(Q_j)}{\varphi_{p}(Q_j)\log^2(\delta/\varphi_{p}(Q_j))}-\dfrac{\epsilon}{\ell}\right)\Omega_p\\ 
& =\sum_{p=0}^{M-1}2\left (\sum_{j=1}^q\Phi_{jp}\right).
\end{align*}
Then, by Theorem \ref{3.4} we have
$$ A\ge C_1\sum_{p=0}^{M-1}2\left (\prod_{j=1}^q\Phi_{jp}\right)^{\frac{1}{\Delta(M-p)}}\Omega_p, $$
for some positive constant $C_1>0$. Let $\Omega_p=h_pdd^c|z|^2$, we have
\begin{align*}
A&\ge 2C_1\sum_{p=0}^{M-1}\left (h_p^{M-p}\prod_{j=1}^q\Phi_{jp}^{1/\Delta}\right)^{\frac{1}{M-p}}dd^c|z|^2\\ 
& \ge C'_1\sum_{p=0}^{M-1}(M-p)\left (h_p^{M-p}\prod_{j=1}^q\Phi_{jp}^{1/\Delta}\right)^{\frac{1}{M-p}}dd^c|z|^2\\ 
&\ge C_2\prod_{p=0}^{M-1}\left (h_p^{M-p}\prod_{j=1}^q\Phi_{jp}^{1/\Delta}\right)^{\frac{2}{M(M+1)}}dd^c|z|^2,
\end{align*}
for some positive constants $C'_1>0,C_2>0$. On the other hand, we note that
\begin{align*}
\prod_{p=0}^{M-1}\Phi_{jp}&=\prod_{p=0}^{M-1}\dfrac{\varphi_{p+1}(Q_j)}{\varphi_{p}(Q_j)}\dfrac{1}{\log^2(\delta/\varphi_{p}(Q_j))}\\ 
& =\dfrac{|F_{0}|^2}{|F_{0}(Q_j)|^2}\prod_{p=0}^{M-1}\dfrac{1}{\log^2(\delta/\varphi_{p}(Q_j))}
\end{align*}
and
$$ \prod_{p=0}^{M-1}h_p^{M-p}=\prod_{p=0}^{M-1}\left (\dfrac{|F_{p-1}|^2|F_{p+1}|^2}{|F_{p}|^4}\right)^{M-p}=\dfrac{|F_{M}|^2}{|F_{0}|^{2(M+1)}},$$
because $\varphi_{0}(Q_j)=|F_{0}(Q_j)|/|F_{0}|,\varphi_{M}(Q_j)=1$. Therefore, we get
$$ A\ge C\left (\dfrac{|F_{0}|^{2(\ell-M-1)}|F_{M}|^2}{\prod_{j=1}^q(|F_{0}(Q_j)|^2\prod_{p=0}^{M-1}\log^2(\delta/\varphi_{p}(Q_j)))^{1/\Delta}}\right)^{\frac{2}{M(M+1)}}dd^c|z|^2.$$
This completes the proof of the theorem.
\end{proof}

Let
$$\hat{h} := \dfrac{\prod_{p=0}^{M-1}|F_{p}|^{2\epsilon}}{\prod_{1\le j\le q,0\le p\le M-1}\log^{2/\Delta}\left(\delta/\varphi_{p}(Q_j)\right)}.$$
Then we have the following important corollary.

\begin{corollary}\label{3.6}
With the above notations, we have
$$dd^c \log \hat{h} \geq C\left (\dfrac{|F_{0}|^{2((q/\Delta)-M-1)}|F_{M}|^2\cdot\hat{h}}{\prod_{p=0}^{M-1}|F_{p}|^{2\epsilon}\prod_{j=1}^q(|F_{0}(Q_j)|^2)^{1/\Delta}}\right)^{\frac{2}{M(M+1)}}dd^c|z|^2,$$
or, equivalently,
\begin{align}\label{3}
\frac{1}{A} \log \hat{h} \geq C h^*, 
\end{align}
where
$$h^* := \left (\dfrac{|F_{0}|^{2((q/\Delta)-M-1)}|F_{M}|^2\cdot\hat{h}}{\prod_{p=0}^{M-1}|F_{p}|^{2\epsilon}\prod_{j=1}^q(|F_{0}(Q_j)|^2)^{1/\Delta}}\right)^{\frac{2}{M(M+1)}}.$$
\end{corollary}

\subsection*{3.3 Second main theorem concerning the Wronskian}
We recall the following formulas in Nevanlinna theory. For a homogeneous polynomial $Q$ in $\C[x_0,\ldots,x_n]$, the proximity function of $f$ with respect to $Q$ is given by
$$m_f(r,Q)=\frac12 \int_{0}^{2\pi} \log \varphi_0(Q)(re^{i\theta})\, \frac{d\theta}{2\pi}.$$
The height function of $F$ is defined by
$$
T_{F_k}(r) = \int_{0}^{r}\frac{dt}{t} \int_{|z|<t}F_k^*\omega
=\int_{|z|<r}\log\frac{r}{|z|} F_k^*\omega.
$$
Let $X_t$ be a Brownian motion on the complex plane $\bbC$. Let $\tau_r := \inf\{t>0,\,|X_t| \geq r\}$ be the stopping time. From the inequality (\ref{1}) with $D_r = \{z \mid |z| < r\}$, we have
$$m_f(r,Q) = \frac{1}{2}\bbE_{0}[\log\varphi_{0}(Q)(X_{\tau_r})],$$
and from the inequality (\ref{2}) with $D_r = \{z \mid |z| < r\}$, we have
$$T_{F_k}(r) = \frac{1}{2}\bbE_{0} \left[ \int_{0}^{\tau_r} e_k(f)(X_s)\,ds \right],$$
where $e_k(f) = \frac{2 F_k^*\omega}{\phi_0}$ and $\phi_0 = \frac{\sqrt{-1}}{2\pi}dz\wedge d\bar{z}= dd^c |z|^2$.

\noindent
\begin{theorem}[Second main theorem with counting function of Wronskian]\label{3.7}
Let $f$ be a holomorphic map from $\C$ into a projective subvariety $V$ of $\P^n(\C)$ with a reduced representation $\tilde f=(f_0,\ldots,f_n)$ and let $\{Q_1,\ldots,Q_q\}$ be a family of $q$ nonzero homogeneous polynomials in $\C[x_0,\ldots,x_n]\setminus I(V)$ which has distributive constant $\Delta$ with respect to $V$. Let $d=lcm(\deg Q_1,\ldots,\deg Q_q)$, $M=H_V(d)-1,$ and $v_0,\ldots,v_M$ is a basis of $I_d(V)$. Assume that $f$ is nondegenerate over $I_d(V)$ and let  $W = W(v_0(\tilde f),\dots,v_M(\tilde f))$ be the Wronskian determinant. Then, for any positive $\epsilon$ and $\delta$, we have
$$\|\,(q-\Delta(M+1+\epsilon)T_f(r)\le\sum_{j=1}^q\frac{1}{\deg Q_j}N_f(r,Q_j)+\Delta N_W(r,0)+\Delta\delta \log r + O(1).$$
\end{theorem}

\begin{proof} Without loss of generality, we may assume that $\deg Q_j=d$ for all $1\le j\le q$. With the notation as in Subsection 3.2, we define the functions
$$\hat{h} := \dfrac{\prod_{p=0}^{M-1}|F_{p}|^{2\epsilon}}{\prod_{1\le j\le q,0\le p\le M-1}\log^{2/\Delta}\left(\delta/\varphi_{p}(Q_j)\right)}$$
and
\begin{align*}
{h}^*&=\left (\dfrac{|F_{0}|^{2(\ell-M-1)}|F_{M}|^2\cdot\hat{h}}{\prod_{p=0}^{M-1}|F_{p}|^{2\epsilon}\prod_{j=1}^q(|F_{0}(Q_j)|^2)^{1/\Delta}}\right)^{\frac{2}{M(M+1)}}\\
&=\left (\dfrac{|F_{0}|^{-2(M+1)}|F_{M}|^2\cdot\hat{h}}{\prod_{p=0}^{M-1}|F_{p}|^{2\epsilon}\prod_{j=1}^q(\varphi_0(Q_j))^{1/\Delta}}\right)^{\frac{2}{M(M+1)}},
\end{align*}
where $\ell=q/\Delta$. Then, from the definition, for $\epsilon'>0$ we have
\begin{align*}
\frac{M(M+1)}{4} \bbE_0[\log {h}^*(X_{\tau_r})]&= \sum_{j=1}^{q}\frac{1}{\Delta}m_f(r,Q_j)- (M+1)dT_f(r) + \bbE_0[|F_M|(X_{\tau_r})]\\
&-\epsilon'\sum_{k=0}^{M-1} \bbE_0[\log|F_k|(X_{\tau_r})] + \frac{1}{2} \bbE_0[\log \hat{h}(X_{\tau_r})].
\end{align*}
Notice that $|F_M|=|W|$. So, from Lemma \ref{2.3} and notice that $\Delta \log \|F_M\| = 0$, we get
$$\bbE_0[\log|F_M|(X_{\tau_r})] = N_W(r,0).$$
Also from Lemma \ref{2.3} and Lemma \ref{3.1}, for $\delta'>0$ we have
$$\|\,\bbE_0[\log|F_k|(X_{\tau_r})] = N_{d_k}(r,0) + T_{F_k}(r) \leq 2(M+1)^2dT_f(r) + O(\log T_f(r)) + \delta'\log r.$$
Hence, by choosing $\epsilon'=\frac{\epsilon}{2(M+1)^2+1}$ and $\delta'=\frac{\delta}{2\epsilon'}$, we get
\begin{align*}
\bigl\|\, \frac{M(M+1)}{4} \bbE_0[\log{h}^*(X_{\tau_r})]& \geq \sum_{j=1}^q\frac{1}{\Delta}m_f(r,Q_j) -(M+1+\epsilon)dT_f(r) + N_W(r,0)\\
&-\frac{\delta}{2}\log r+\frac{1}{2}\bbE_0[\log \hat{h}(X_{\tau_r})].
\end{align*}
We need to estimate $\bbE_0[\log{h}^*(X_{\tau_r})]$. From Jensen's inequality,
$$\bbE_0[\log{h}^*(X_{\tau_r})] \leq \log \bbE_0[{h}^*(X_{\tau_r})].$$
From Lemma \ref{2.4} and the inequality (\ref{3}), we have
\begin{align*}
\|\, \log \bbE_0[{h}^*(X_{\tau_r})]&\leq (1+\delta)^2 \log^+ \bbE_0 \left[ \int_{0}^{\tau_r} h^*(X_s)\, ds \right] + \delta \log r\\
&\leq (1+\delta)^2 \log^+ \bbE_0 \left[ \int_{0}^{\tau_r} (\Delta\log\hat{h})(X_s)\, ds \right] + \delta \log r+O(1).
\end{align*}
Using Lemma \ref{2.3} again with $\phi = \log \hat{h}$ and $T = \tau_r$, and noting that $N_{\hat{h}}(r,0) = 0$, we get
$$
\bbE_0[\log \hat{h}(X_{\tau_r})] \ge\bbE_0 \left[ \int_{0}^{\tau_r} (\Delta \log \hat{h})\, ds \right].
$$
Hence
$$\|\, \log \mathrm{E}_0\left[h^*\left(X_{\tau_r}\right)\right] \leqslant(1+\delta)^2 \log { }^{+} \mathrm{E}_0\left[\log \hat{h}\left(X_{\tau_r}\right)\right]+\delta \log r+O(1).$$
Thus
\begin{align*}
\|\,  \sum_{j=1}^q\frac{1}{\Delta}m_f(r,Q_j)& -(M+1+\epsilon)dT_f(r) + N_W(r,0)-\frac{\delta}{2}\log r+\frac{1}{2}\bbE_0[\log \hat{h}(X_{\tau_r})]\\
&\leq (1+\delta)^2 \log^+ \bbE_0 \left[ \int_{0}^{\tau_r} (\Delta \log \hat{h})(X_s) \, ds \right] + \delta \log r + O(1).
\end{align*}
Notice that for any constant $C>0$, the quantity $C \log \bbE_0[\hat{h}(X_{\tau_r})] - \bbE_0[\log \hat{h}(X_{\tau_r})]$ is bounded from above (for sufficiently large
$r$). Then we have
$$\|\,\sum_{j=1}^q\frac{1}{\Delta}m_f(r,Q_j) + N_W(r,0) \leq (M+1+\epsilon)dT_f(r) + \delta \log r + O(1),$$
i.e.,
$$\|\,(q-\Delta(M+1+\epsilon))T_f(r)\le\sum_{j=1}^q\frac{1}{\deg Q_j}N_f(r,Q_j)-\frac{\Delta}{d}N_W(r,0)+\Delta\delta \log r + O(1).$$
We get the desired second main theorem.
\end{proof}

\section{Proof of main results}

We shall need the following lemma.

\begin{lemma}[{see \cite[Lemma 3.1]{Q21}}]\label{4.1}
Let $t_0,t_1,\ldots,t_n$ be $n+1$ integers such that $1=t_0<t_1<\cdots <t_n$, and let $\Delta =\underset{1\le s\le n}\max\dfrac{t_s-t_0}{s}$. 
Then for every $n$ real numbers $a_0,a_1,\ldots,a_{n-1}$  with $a_0\ge a_1\ge\cdots\ge a_{n-1}\ge 1$, we have
$$ a_0^{t_1-t_0}a_1^{t_2-t_1}\cdots a_{n-1}^{t_{n}-t_{n-1}}\le (a_0a_1\cdots a_{n-1})^{\Delta}.$$
\end{lemma}

\begin{proof}[Proof of Theorem \ref{1.1}]
Let $Q_1,\ldots,Q_q$ be the homogeneous polynomials defining $D_1,\ldots,D_q$ respectively. Without loss of generality, we may assume that $\deg Q_j=d$ for all $1\le j\le q$.  Using the notation in the proof of Theorem \ref{3.7}, we have
\begin{align}\label{4}
\|\,(q-\Delta(M+1+\epsilon))T_f(r)\le\sum_{j=1}^q\frac{1}{ d}N_f(r,Q_j)-\frac{\Delta}{d}N_W(r,0)+\Delta\delta \log r + O(1),
\end{align}
where $M=H_V(d)-1.$ Let $\tilde f=(f_0,\ldots,f_n)$ be a reduced representation of $f$.We now establish the following inequality for the counting functions of divisors:
 $$\sum_{i=1}^qN_{Q_i(f)}(r)-\Delta N_{W}(r)\le \sum_{i=1}^qN^{[M]}_{Q_i(\tilde f)}(r).$$

Indeed, let $z$ be a zero of some $Q_i(\tilde f)$. We may assume that 
$$\nu^0_{Q_1(\tilde f)}(z)\ge \nu^0_{Q_2(\tilde f)}(z)\ge\ldots\ge\nu^0_{Q_\ell(\tilde f)}(z)>0=\nu^0_{Q_{\ell+1}(\tilde f)}(z)=\cdots=\nu^0_{Q_q(\tilde f)}(z).$$
Note that if $q>\Delta (M+1)$ then $\frac{\ell}{k}\le\Delta<\frac{q}{M+1},$ i.e., $\ell<\frac{kq}{M+1}\le q-1.$ Then, there are indices $i_0,\ldots,i_p,i_{p+1}$ such that $1=i_0<i_1<i_2<\cdots<i_p\le\ell<i_{p+1}$ and
$$\rank\{Q_i;1\le i\le t\}=s\ \forall\ i_{s}\le t\le i_{s+1}-1,\, 0\le s\le p.$$
Since $\{Q_i\}_{i=1}^q$ has distribution constant $\Delta$, we have 
$$\Delta\ge\frac{i_s-1}{k-\dim (\bigcap_{1\le j\le i_s-1}D_{i_j}\cap V)}\ge\frac{i_s-i_0}{\rank\{Q_i;1\le i\le i_s-1\}}=\frac{i_s}{s}\, \forall 1\le s\le p+1.$$

We set $t_i=\nu^0_{Q_i(\tilde f)}(z)$ for $i=1,\ldots,\ell$. Then, by Lemma \ref{4.1} we have
\begin{align*}
\sum_{i=1}^q\max\{0,\nu^0_{Q_i(\tilde f)}(z)-M\}
&\le\sum_{s=0}^p(i_{s+1}-i_s)\max\{0,\nu^0_{Q_{i_s}(\tilde f)}(z)-M\}\\
&\le\Delta(\sum_{s=0}^p\max\{0,\nu^0_{Q_{i_s}(\tilde f)}(z)-M\}).
\end{align*}
We take $M-p$ homogeneous polynomials $P_{p+1},\ldots,P_M$ of degree $d$ in variables $(x_0,\ldots,x_n)$ such that $([Q_{i_0}],\ldots,[Q_{i_p}],[P_{p+1}],\ldots,[P_M])$ is a basis of $I_d(V)$. Then, by usual argument in Nevanlinna theory, we have
\begin{align*}
\nu^0_W(z)&=\nu^0_{W(Q_{i_0}(\tilde f),\ldots,Q_{i_p}(\tilde f),P_{p+1}(\tilde f),\ldots, P_M(\tilde f))}(z)\\
&\ge\sum_{s=0}^p\max\{0,\nu^0_{Q_{i_s}(\tilde f)}(z)-M\}+\sum_{s=p+1}^M\max\{0,\nu^0_{P_s(\tilde f)}(z)-M\}\\
&\ge\sum_{s=0}^p\max\{0,\nu^0_{Q_{i_s}(\tilde f)}(z)-M\}.
\end{align*}
Therefore,
\begin{align*}
\sum_{i=1}^q\nu^0_{Q_i(f)}(z)-\Delta \nu^0_{W}(z)&\le \sum_{i=1}^q\nu^0_{Q_i(f)}(z)-\Delta\sum_{s=0}^p\max\{0,\nu^0_{Q_{i_s}(\tilde f)}(z)-M\}\\
&\le \sum_{i=1}^q\nu^0_{Q_i(f)}(z)-\sum_{i=1}^q\max\{0,\nu^0_{Q_i(\tilde f)}(z)-M\}\\
&=\sum_{i=1}^q\min\{M,\nu^0_{Q_i(\tilde f)}(z)\}.
\end{align*}
This yields that
 $$\sum_{i=1}^qN_{Q_i(f)}(r)-\Delta N_{W}(r)\le \sum_{i=1}^qN^{[M]}_{Q_i(\tilde f)}(r).$$
Combining this inequality and (\ref{4}), we obtain
\begin{align*}
\|\,(q-\Delta(M+1+\epsilon))T_f(r)&\le\sum_{j=1}^q\frac{1}{ d}N^{[M]}_f(r,Q_j)+\Delta\delta \log r + O(1)\\
&\le\sum_{j=1}^q\frac{1}{\deg D_j}N^{[M]}_f(r,D_j)+\Delta\delta \log r + O(1).
\end{align*}
Hence, the theorem is proved.
\end{proof}

\begin{proof}[Proof of Theorem \ref{1.2}]
Assume that $\tilde f=(f_0,\ldots,f_n)$ and $\tilde g=(g_0,\ldots,g_n)$ are reduced representations of $f$ and $g,$ respectively. Let $Q_1,\ldots,Q_q$ be the homogeneous polynomials defining$D_1,\ldots,D_q$ respectively. Without loss of generality, we may assume that $\deg Q_j=d$ for all $1\le j\le q$.

Suppose that $f\ne g$. Then there exist two indices $s,t$ with $0\le s<t\le n$ such that $ H:=f_sg_t-f_tg_s\not\equiv 0.$
By the definition of the characteristic function and by the Jensen formula, we have
\begin{align}\label{5}
\begin{split}
N_H(r)&=\int_{0}^{2\pi}\log |f_s(re^{i\theta})g_t(re^{i\theta})-f_t(re^{i\theta})g_s(re^{i\theta})|\frac{d\theta}{2\pi}\\ 
& \le \int_{0}^{2\pi}\log \|f(re^{i\theta})\|\frac{d\theta}{2\pi} + \int_{0}^{2\pi}\log \|g(re^{i\theta})\|\frac{d\theta}{2\pi} \\
&=T_f(r)+T_g(r).
\end{split}
\end{align}

(a) By the assumption of the theorem, we have $H=0$ on $\bigcup_{i=1}^qf^{-1}(D_i)\cup g^{-1}(D_i)$. 
If $z$ is a zero of exactly $\ell$ functions $Q(\tilde f)$ then $\Delta\ge\frac{\ell}{k}$, i.e., $\ell\le\Delta k.$
Therefore, 
$$\nu^0_H(z)\ge \Delta k\sum_{i=1}^q\min\{1,\nu^0_{Q_i(f)}\}.$$
It follows that
\begin{align}\label{6}
N_H(r)\ge \Delta k\sum_{i=1}^qN^{[1]}_{Q_i(f)}(r).
\end{align}
Combining (\ref{5}) and (\ref{6}), we have
$$ T_f(r)+T_g(r) \ge \Delta k\sum_{i=1}^qN^{[1]}_{Q_i(f)}(r).$$
Similarly, we have
$$ T_f(r)+T_g(r) \ge \Delta k\sum_{i=1}^qN^{[1]}_{Q_i(g)}(r).$$
Summing both sides of the above inequalities, we get
\begin{align*}
 2(T_f(r)+T_g(r))& \ge \Delta k\left(\sum_{i=1}^qN^{[1]}_{Q_i(f)}(r)+\sum_{i=1}^qN^{[1]}_{Q_i(g)}(r)\right).
\end{align*}
This yields that
\begin{align*}
\frac{2}{\Delta k}(T_f(r)+T_g(r))&\ge\sum_{i=1}^qN^{[1]}_{Q_i(f)}(r)+\sum_{i=1}^qN^{[1]}_{Q_i(g)}(r)\\
&\ge \sum_{i=1}^q\dfrac{1}{H_V(d)-1}N^{[H_V(d)-1]}_{Q_i(f)}(r)+\sum_{i=1}^q\dfrac{1}{H_V(d)-1}N^{[H_V(d)-1]}_{Q_i(g)}(r)\\
&\ge \dfrac{d}{H_V(d)-1}(q-\Delta (H_V(d)+\epsilon))(T_f(r)+T_g(r))+o(T_f(r)+T_g(r)).
\end{align*}
Letting $r\longrightarrow +\infty$ and then letting $\epsilon\longrightarrow 0$, we get 
$$2\Delta k\ge \dfrac{d}{H_V(d)-1}(q-\Delta H_V(d)),$$
$$\text{i.e., } q\le \Delta\left(\frac{2k(H_V(d)-1)}{d}+ H_V(d)\right).$$
This is a contradiction.
Hence $f=g$.

(b) Since $H=0$ on $\bigcup_{i=1}^qf^{-1}(D_i)\cup g^{-1}(D_i)$, we have
$$\nu^0_H(z)\ge \sum_{i=1}^q\min\{1,\nu^0_{Q_i(f)}\}.$$
It follows that
$$N_H(r)\ge \sum_{i=1}^qN^{[1]}_{Q_i(f)}(r).$$
Combining (\ref{5}) and this inequality, we obtain
$$ T_f(r)+T_g(r) \ge \sum_{i=1}^qN^{[1]}_{Q_i(f)}(r)\, \text{ and }\, T_f(r)+T_g(r) \ge \sum_{i=1}^qN^{[1]}_{Q_i(g)}(r).$$
Summing-up both sides of the above two inequalities and applying Theorem \ref{1.1} we have
\begin{align*}
2(T_f(r)+T_g(r))&\ge\sum_{i=1}^qN^{[1]}_{Q_i(f)}(r)+\sum_{i=1}^qN^{[1]}_{Q_i(g)}(r)\\
&\ge \sum_{i=1}^q\dfrac{1}{H_V(d)-1}N^{[H_V(d)-1]}_{Q_i(f)}(r)+\sum_{i=1}^q\dfrac{1}{H_V(d)-1}N^{[H_V(d)-1]}_{Q_i(g)}(r)\\
&\ge \dfrac{d}{H_V(d)-1}(q-\Delta (H_V(d)+\epsilon))(T_f(r)+T_g(r))+o(T_f(r)+T_g(r)).
\end{align*}
Letting $r\longrightarrow +\infty$ and then letting $\epsilon\longrightarrow 0$, we get 
$$2 \ge \dfrac{d}{H_V(d)-1}(q-\Delta H_V(d)),$$
$$\text{i.e., } q\le \frac{2(H_V(d)-1)}{d}+\Delta H_V(d).$$
This is a contradiction.
Hence $f=g$.
\end{proof}

\section*{Data availability}
Data sharing not applicable to this article as no datasets were generated or analyzed during the current study.

\section*{Disclosure statement}
No potential conflict of interest was reported by the author(s).

\end{document}